\documentclass[12pt]{article}

\usepackage{amssymb}
\usepackage{amsmath}
\usepackage{amsthm}
\usepackage{bm}
\usepackage{graphicx}
\usepackage{subcaption}
\usepackage{hyperref}

\newtheorem{theorem}{Theorem}
\newtheorem{proposition}{Proposition}
\newtheorem{definition}{Definition}
\newtheorem{lemma}{Lemma}
\newtheorem{corollary}{Corollary}
\theoremstyle{remark}

\title{On a Question of Mirzakhani and Vondr\'ak concerning Sperner Labelings of the Simplex-Lattice Hypergraph}

\author{Ognjen Papaz\thanks{Faculty of Philosophy, University of East Sarajevo, Bosnia and Herzegovina {\tt ognjen.papaz@ff.ues.rs.ba}}
 \and Du\v sko Joji\'c\thanks{Faculty of Science, University of Banja Luka, Bosnia and Herzegovina {\tt dusko.jojic@pmf.unibl.org}}}
\date{}

\begin{document}

\maketitle

\begin{abstract}
Mirzakhani and Vondr\'ak asked whether the simplex-lattice hypergraph
admits a Sperner-admissible labeling in which every hyperedge receives at most
two colors. We answer this question affirmatively by introducing the
\emph{dominant labeling}, in which every lattice point is labeled by its
unique \emph{dominant index}. In their hybrid model of fair division, 
this yields preference functions under which at most
two players are nearly satisfied. We determine the number of vertices of each
color and the number of hyperedges using a given pair of colors; these counts
are expressed through the entries of Catalan's triangle.
We also characterize the color sets of hyperedges of the other hypergraphs that
arise from the edgewise triangulation of the dilations of the standard unit simplex.
\end{abstract}

\section{Introduction}

Sperner's lemma \cite{Sperner28} is one of the classical results in discrete geometry and combinatorial topology. Suppose we are given a labeling of the vertices of a triangulation of a $(k-1)$-simplex with $k$ colors such that no two vertices lying on complementary faces of the simplex share the same color. Sperner's lemma states that, under this condition, there exists a fully labeled simplex in the triangulation. Sperner used it in his thesis to prove the Invariance of Domain theorem which was initially proven by Brouwer using his fixed point theorem. Later, the KKM theorem (Knaster-Kuratowski-Mazurkiewicz) \cite{KKM} established the equivalence between Sperner's lemma and Brouwer fixed-point theorem, see \cite{DU,deL}. In recent decades, Sperner's lemma has found its applications serving as a backbone for fair division, see \cite{Su,Frick,MSu}.

The task of fair division is to divide a resource among some agents in a manner that satisfies each of them. The classical setting \cite{Storm,Woodall} considers a rectangular cake placed along $[0,1]$ and sliced perpendicular to this axis. Suppose $k$ persons are waiting for a piece of cake. If a particular division of the cake is specified by the lengths of its slices, then the set of all possible divisions is represented by the standard unit simplex $\Delta_{k-1}=\{\bm{x}\geq \bm{0}:\sum_{i=1}^kx_i=1\}$. Each person $i$ has a preference function $p_i:\Delta_{k-1}\to 2^{[k]}\setminus\{\emptyset\}$ that specifies their preferred pieces under a division. For $i,j\in[k]$, let $C_i^j=\{\bm x\in \Delta_{k-1}:j\in p_i(\bm x)\}$ be the set of all divisions of the cake in which the $j$-th piece is preferred by player $i$. A division $\bm x$ is envy-free if there exists a permutation $\pi\in\mathbb S_k$ such that $\bm x\in\bigcap_{i=1}^kC_i^{\pi(i)}$.

Assuming that the preference sets $C_i^j$ are closed and everyone is hungry, meaning that $j\in p_i(\bm x)$ implies $x_j>0$, an envy-free division exists \cite{Storm,Woodall}. In \cite{Gale}, Gale proved the colorful version of the KKM theorem, which is precisely equivalent to the existence of a an envy-free division; according to \cite{Frick}, this equivalence was first observed by Musin \cite[Sec.~3]{Musin}. Su \cite{Su} proved the existence of a fair division using Sperner's lemma. In \cite{Bapat}, Bapat established the colorful version of Sperner's lemma, providing a discrete analogue to the colorful KKM theorem. We refer to \cite{Frick} for a detailed overview of these results.

Suppose now that a resource is not continuously divisible but comes in $q$ mutually equal pieces. If we attempt to find an envy-free division $\bm a=(a_1,a_2,\ldots,a_k)$, where the components $a_i$ are integers summing up to $q$, such a division is not guaranteed to exist. Mirzakhani and Vondr{\'{a}}k \cite{MV} consider a variant of the problem that interpolates between continuous and discrete settings: $q-1$ units of  the resource are divided discretely and the remaining single unit continuously. For each discrete division of $q-1$, the continuous divisions of the remaining piece are represented by $\Delta_{k-1}$. Therefore, all possible divisions described in this way form a subset $\Delta_{k,q}^{Id}$ of $\Delta_{k,q}=q\Delta_{k-1}$. More precisely, $\Delta_{k,q}^{Id}$ is the union of the family of simplices $\bm a+\Delta_{k-1}$ over all $\bm a\in\mathbb Z^k_+$ such that $\sum_{i=1}^ka_i=q-1$ (see Figure \ref{prostori}). We will also view $\Delta_{k,q}^{Id}$ as the simplicial complex generated by $\{\bm a+\Delta_{k-1}\}$.

\begin{figure}[htb]
\centering
\makebox[\textwidth][c]{
\raisebox{0.8cm}{\includegraphics[width=0.4\textwidth]{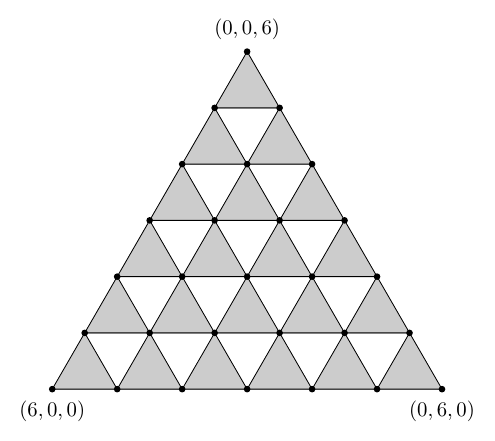}}\hfill
\includegraphics[width=0.55\textwidth]{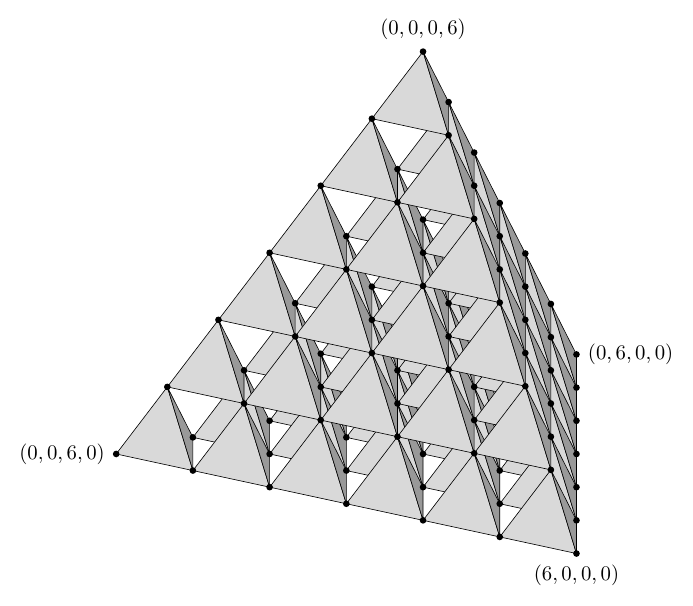}
}
\caption{The complexes $\Delta_{3,6}^{Id}$ and $\Delta_{4,6}^{Id}$}
\label{prostori}
\end{figure}

\begin{definition}[{\cite[Definition 7.2.]{MV}}]
Suppose $p_i:\Delta_{k,q}^{Id}\to 2^{[k]}\setminus\{\emptyset\},i \in[k]$ are preference functions. A player $i$ is \emph{nearly satisfied} with the $j$-th piece in an integral division $\bm a,\sum_{t=1}^k a_t=q-1$, if there is a division $\bm x\in \Delta_{k-1}$ such that $j\in p_i(\bm a+\bm x)$.
\end{definition}

According to \cite[Corollary 7.1]{MV}, any Sperner's coloring of the vertices of the simplices $\bm a +\Delta_{k-1}$, using at most $n$ colors per simplex, induces a preference functions $p_i: \Delta_{k,q}^{Id} \to 2^{[k]}\setminus\{\emptyset\}$, under which, for any integral division $(a_1, a_2, \ldots, a_k)$ with $\sum_{i=1}^k a_i = q-1$, at most $n$ players are nearly satisfied.

In \cite[Lemma 4.1]{MV}, the authors found a Sperner's coloring of the vertices of  the simplices $\bm a +\Delta_{k-1}$ using at most four colors per simplex, and left it as an open question whether there exists a Sperner's  coloring using at most two colors per simplex. In the following section we settle this question by providing such a coloring.

\section{Dominant labeling}
We follow the notation from \cite{MV}. For natural numbers $q,k\geq 2$ let
$$\Delta_{k,q}=\{\bm{x}=(x_1,x_2,\ldots,x_k)\in \mathbb{R}^k
:x_i\geq 0,x_1+x_2+\cdots+x_k=q\}$$ denote a $(k-1)$-dimensional
simplex. Integral points in $\Delta_{k,q}$ form the set
$$V_{k,q}=\mathbb{Z}^k\cap \Delta_{k,q}=\{\bm{v}=(v_1,v_2,
\ldots,v_k): v_i\in \mathbb{N}_0, v_1+v_2+\cdots+v_k=q\}.$$

The \emph{Simplex-Lattice Hypergraph} $H_{k,q}=(V_{k,q},E_{k,q})$
is a $k$-uniform hypergraph with $V_{k,q}$ as the vertex set.
 The hyperedges (simplices) are indexed
by the points in $V_{k,q-1}$. For each $\bm{a}\in V_{k,q-1}$ let
$e(\bm{a})=\{\bm{a}+e_1, \bm{a}+e_2,\ldots,\bm{a}+e_k\}$, that is
$e(\bm{a})$ contains vertices of the simplex $\bm a+\Delta_{k-1}$.
The set of hyperedges of $H_{k,q}$ is
$E_{k,q}=\{e(\bm{a})\}_{\bm{a}\in V_{k,q-1}}$.

\noindent The \emph{Sperner-admissible labeling} of the hypergraph
$H_{k,q}$ is a map $\ell:V_{k,q}\to [k]$ such that $\ell(\bm v)=i$
implies $v_i>0$.
\begin{definition} We say that $i\in[k]$ is a dominant index
of a vertex $\bm{x}=(x_1,x_2,\ldots,x_k)\in V_{k,q}$ if
\begin{gather}
x_i\geq 1, x_i+x_{i+1}\geq 2,\ldots,x_i+x_{i+1} +\cdots+x_k\geq k-i+1,\ \textrm{and}\label{L:bojaigeqprvidio}\\
x_{i-1}<1,x_{i-1}+x_{i-2}<2,\ldots,x_{i-1}+x_{i-2}+\cdots+x_1<i-1.\label{L:bojai<drugidio}
\end{gather}
\end{definition}
\begin{proposition} If $q\geq k$, then each vertex
$\bm{x}\in V_{k,q}$ has exactly one
dominant index.
\end{proposition}
\begin{proof} We use induction on $k$. The case $k=2$ is trivial.
If $1$ is not the dominant index of
$\bm{x}=(x_1,x_2,\ldots,x_k)\in V_{k,q}$, let $j\in[k]$ denote
the smallest index such that $x_1+x_2+\cdots+x_{j}<j$. In that
case we have $x_1+x_2+\cdots+x_{j-1}=j-1$ and $x_{j}=0$. Therefore
$x_{j+1}+\cdots+x_k=q-j+1>k-j$, and we apply the inductive
assumption on $\overline{\bm{x}}=(x_{j+1},\ldots,x_k)\in
V_{k-j,q-j+1}$. If $t\in [k-j]$ is a dominant index of
$\overline{\bm{x}}$, then
\begin{gather*}
x_{j+t}\geq 1, x_{j+t}+x_{j+t+1}\geq 2,\ldots, x_{j+t}+\cdots +x_k\geq k-j-t+1,\ \textrm{and}\\
x_{j+t-1}<1, x_{j+t-1}+x_{j+t-2}<2, \ldots, x_{j+t-1}+\cdots+ x_{j+1}<t-1
\end{gather*}
Since $x_1+x_2+\cdots+x_{i-1}\geq i-1$ for all $i\leq j$, we have, for
all $i\in [j]$,
\[x_i+x_{i+1}+\cdots+x_j+x_{j+1}+\cdots+x_{j+t-1}<t+j-i.\]
Therefore, $j+t$ is a dominant index of $\bm{x}$.

\noindent If $\bm{x}=(x_1,x_2,\ldots,x_k)\in V_{k,q}$ has two
dominant indices $i,j\ (i<j)$, then the dominance of $i$ implies
$x_i+x_{i+1}+\cdots+x_{j-1}\geq j-i$, while the dominance of $j$ implies
$x_i+x_{i+1}+\cdots+x_{j-1}< j-i$, which is a contradiction.
\end{proof}
The above proposition enables us to define \textit{the dominant
labeling}.
\begin{definition}Assume that $q\geq k$. The dominant labeling
of the points in $V_{k,q}$ is a map $\ell:V_{k,q}\rightarrow[k]$
that sends each vertex $\bm{x}$ to its unique dominant index.
\end{definition}

Since $\ell(\bm{x})=i$ implies $x_i>0$, the dominant labeling is a
Sperner-admissible labeling of $H_{k,q}$.

The dominant labeling has a nice geometric interpretation.
The intersections of the half-spaces specified within (\ref{L:bojaigeqprvidio}) and
(\ref{L:bojai<drugidio}) form a partition of 
\[A_{k,q}=\{\bm{x}\in \mathbb{R}^k:x_1+x_2+\cdots+x_{k}=q\}\cong
\mathbb{R}^{k-1}\] 
into $k$ pointed cones with the common apex at
the point $(1,1,\ldots,1,q-k+1)$. A point in $V_{k,q}$ is colored 
with $i$ under the dominant labeling if and only if it belongs to the
cone containing $q\bm e_i$ (see Figure \ref{konusi}).

\begin{figure}[htb]
\centering
\makebox[\textwidth][c]{
\raisebox{0.8cm}{\includegraphics[width=0.4\textwidth]{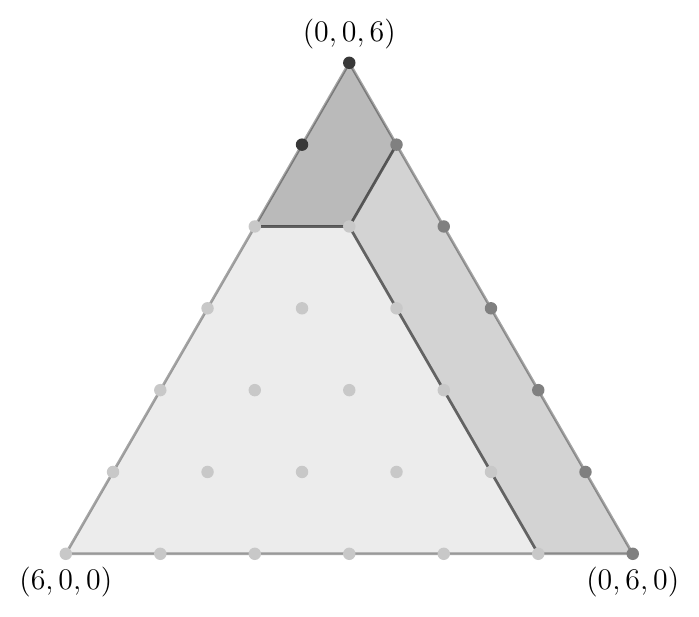}}\hfill
\includegraphics[width=0.55\textwidth]{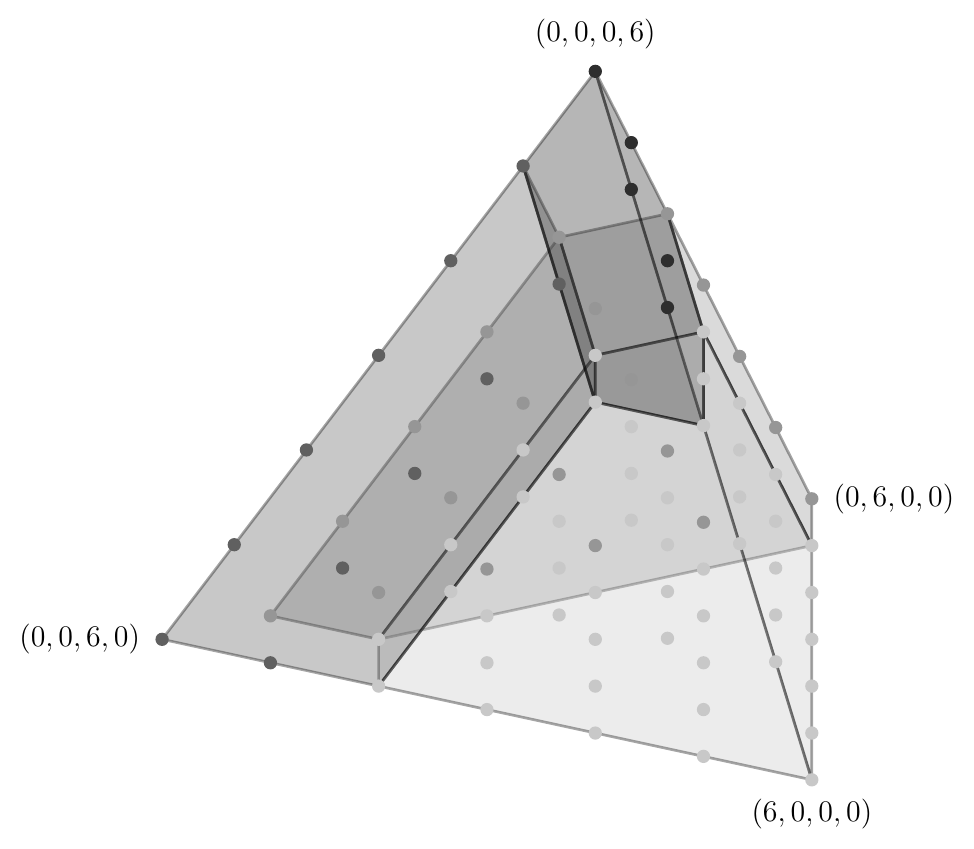}
}
\caption{A geometric interpretation of the dominant
labeling}
\label{konusi}
\end{figure}

The following theorem answers the question raised in \cite{MV}.
\begin{theorem}\label{T:justtwocolors}
If $q>k$, then under the dominant labeling every hyperedge of $H_{k,q}$ uses
at most two colors.
\end{theorem}
\begin{proof} We consider an arbitrary
hyperedge $e(\bm{v})=\{\bm{v}+e_1,\bm{v}+e_2, \ldots,\bm{v}+e_k\}$
defined by $\bm{v}=(v_1,v_2,\ldots,v_k)\in V_{k,q-1}$. The
dominant labeling is well-defined for vertices from $V_{k,q-1}$. 
Assume that $\ell(\bm{v})=i$, i.e.
\begin{gather*}
v_i\geq 1, v_i+v_{i+1}\geq 2,\ldots,v_i+v_{i+1} +\cdots+v_k\geq k-i+1, \textrm{and}\\
v_{i-1}<1, v_{i-1}+v_{i-2}<2,\ldots, v_{i-1}+v_{i-2}+\cdots+v_1<i-1.
\end{gather*}
All of the above inequalities are satisfied by
$\bm{v}+e_k,\bm{v}+e_{k-1},\ldots, \bm{v}+e_i$.
Therefore,
\[\ell(\bm{v}+e_k)=\ell(\bm{v}+e_{k-1})=\cdots
=\ell(\bm{v}+e_i)=\ell(\bm{v})=i.\]
If $i>1$, we consider the
smallest index $j<i$ such that 
\begin{equation}\label{E:najmanji
j} v_j+v_{j+1}+\cdots+v_{i-1}=i-j-1.
\end{equation} 
Since for each index  $p, j\leq
p<i-1$ we have that $v_{p+1}+v_{p+2}+\cdots+v_{i-1}<i-p-1$, it
follows from (\ref{E:najmanji j}) that
\[v_j+v_{j+1}+\cdots+v_p\geq p-j+1.\] 
Also for $t\geq i$, we obtain that
\[v_j+v_{j+1}+\cdots+v_{i-1}+v_i+\cdots+v_t\geq
(i-j-1)+(t-i+1)=t-j.\] 
Therefore, we have that
\[\ell(\bm{v}+e_{i-1})=\ell(\bm{v}+e_{i-2})=\cdots
=\ell(\bm{v}+e_j)=j.\] 
If $j>1$, then the rest of the vertices of
$e(\bm{v})$ are again colored with $i$, i.e.
\[\ell(\bm{v}+e_{j-1})=\ell(\bm{v}+e_{j-2})=\cdots
=\ell(\bm{v}+e_1)=i.\]
\end{proof}

For $\bm x\in\Delta_{k,q}^{Id}$, let $\mathrm{supp}(\bm x)$ be the set of vertices of the carrier simplex of $\bm x$.
\begin{corollary}
If $p_i(\bm x)=\ell(\mathrm{supp}(\bm x))$ for every $i\in[k]$ and every $\bm x\in\Delta_{k,q}^{Id}$, then for any integral division $\bm a,\sum_{i=1}^ka_i=q-1$, at most two players are nearly satisfied.
\end{corollary}

\section{Counting vertices and hyperedges by color}

We determine the number of vertices labeled with $i\in[k]$ under
the dominant labeling $\ell$, as well as the number of hyperedges
of $H_{k,q}$ that use exactly the labels
$a,b\in[k]$. For all integers $1\leq i \leq k<q$ we define
\[c^{i}_{k}(q)=|\{\bm{x}\in V_{k,q}:\ell(\bm{x})=i\}|.\]
Let $\mathcal{M}_n^t$ denote the set of sequences
$(x_1,x_2,\ldots,x_n)\in \mathbb{N}^n_0$ such that
\[x_1<1,x_1+x_2<2,\ldots,x_1+x_2+\cdots+x_{n-1}<n-1
\textrm{ and } x_1+x_2+\cdots+x_n=t<n.\] The numbers
$M_n^t=|\mathcal{M}_n^t|$ satisfy the recursive relations
\[M_n^t=M_{n-1}^t+M_{n}^{t-1}\textrm{, and }
M_n^0=1\textrm{ for all }n\in \mathbb{N}.\] Therefore, we can
recognize these numbers as shifted entries of Catalan's triangle
(see the sequence A009766 in \cite{Sloan} and \cite{Bai}).
\begin{proposition}\label{P:Mkaorazlika} For all $0< t<n$ we have that
$$M_n^t=T(n-1,t)=\binom{n+t-1}{t}-\binom{n+t-1}{t-1}.$$
\end{proposition}
The last two entries in each row of Catalan's triangle are Catalan numbers;
that is, $M_n^{n-2}=M_n^{n-1}=C_{n-1}$ for every $n\geq 2$.

\noindent We also denote by $\mathcal{V}_m^j$ the set of
sequences $(x_1,x_2,\ldots,x_m)\in \mathbb{N}_0^m$ such that
\[
x_1\geq 1, x_1+x_2\geq 2, \ldots, x_1+\cdots+x_{m-1}\geq m-1, x_1+\cdots+x_m=j\geq m,
\]
and let $V_m^j=|\mathcal{V}_m^j|$.
\begin{proposition}For all $m,j$ such that $1\leq m\leq j$ we have that
$V_m^j=M_{j+1}^{m-1}$.
\end{proposition}
\begin{proof} A sequence $(y_1,y_2,\ldots,y_m)\in \mathcal{V}_m^j$
defines a lattice path from $(0,0)$ to $(m-1,j)$ in $\mathbb{Z}^2$
that never descends below $y=x$. Specifically, for
$i=1,2,\ldots,m-1$ after $y_i$ steps up, we take one step right.
We finish with $y_m$ steps up.

Similarly, a sequence $(x_1,x_2,\ldots,x_{j+1})\in
\mathcal{M}_{j+1}^{m-1}$ can be viewed as a lattice path from
$(0,0)$ to $(m-1,j)$ in $\mathbb{Z}^2$ that never descends below
$y=x$. Specifically, for $i=1,2,\ldots,j$ after $x_i$ steps right,
we take one step up. We finish with $x_{j+1}$ steps right.
\end{proof}
The above argument also gives a bijective proof of Proposition
\ref{P:Mkaorazlika}.
\begin{theorem}For all $k,q\in\mathbb{N}$ with $q>k$, we have
\[c_k^1(q)=V_k^q\ \textrm{and}\
c_k^i(q)=\sum_{j=0}^{i-2}M_{i-1}^j V_{k-i+1}^{q-j}\ \textrm{for}\
1<i\leq k.\]
\end{theorem}
\begin{proof}
A vertex $\bm x\in V_{k,q}$ is labeled with 1 if and only if $\bm x\in \mathcal V_k^q$. 
For $i>1$, the map
\[(x_1,\ldots,x_{i-1},x_i,\ldots,x_k) \mapsto
\big((x_{i-1},\ldots,x_2,x_1),(x_{i},x_{i+1},\ldots,x_k)\big)\]
is a bijection between the vertices of $V_{k,q}$ labeled with $i$ and
$\coprod_{j=0}^{i-2}\mathcal{M}_{i-1}^j\times
\mathcal{V}_{k-i+1}^{q-j}$, see (\ref{L:bojaigeqprvidio}) and
(\ref{L:bojai<drugidio}).
\end{proof}
\begin{corollary} For all $k \in N$ we have that
\[c_k^1(k+1)=V_k^{k+1}=\binom{2k}{
k-1}-\binom{2k}{k-2}.\] For all $q>k$ we have that
$c_k^k(q)=C_{k-1}$, i.e., $c_k^k(q)$ does not depend on $q$.
\end{corollary}

It follows from the proof of Theorem \ref{T:justtwocolors} that
the hyperedge $e(\bm{v})=\{\bm{v}+e_1,\bm{v}+e_2,
\ldots,\bm{v}+e_k\}$ is monochromatic precisely when
$\ell(\bm{v})=1$. Therefore, the number of monochromatic
hyperedges in the dominant labeling of $H_{k,q}$ is $c_k^1(q-1)$.
\begin{proposition} For $1\leq a<b\leq k$, the number of hyperedges of $H_{k,q}$
 that use exactly the labels $a$ and $b$ is $C_{b-a-1}c^a_{k-b+a}(q-b+a)$.
\end{proposition}
\begin{proof}
Consider an arbitrary vertex $\bm{v}=(v_1,v_2,\ldots,v_k)\in
V_{k,q-1}$. From the proof of Theorem \ref{T:justtwocolors} it follows that
the hyperedge $e(\bm{v})$ uses exactly the colors $a$ and
$b$ under the dominant labeling if and only if:
\begin{enumerate}
    \item $v_a+v_{a+1}+\cdots+v_{b-1}=b-a-1$ and
    \[v_{b-1}<1,v_{b-1}+v_{b-2}<2,\ldots,
    v_{b-1}+v_{b-2}+\cdots+v_{a+1}<b-a-1;\]
    \item $(v_1,v_2,\ldots,v_{a-1},v_b,\ldots,v_k)\in V_{k-b+a,q-b+a}$
     is colored with $a$ in the dominant labeling with $k-b+a$
     colors.
\end{enumerate}
\end{proof}

\section{Hypergraphs defined by permutations}

The hyperedges of $H_{k,q}$, viewed as simplices, do not cover the
entire simplex $\Delta_{k,q}$. In \cite{MV}, Mirzakhani and
Vondr{\'{a}}k describe a triangulation of $\Delta_{k,q}$ with the vertex set
$V_{k,q}$ that contains $\Delta^{Id}_{k,q}$ as a subcomplex. They
use the edgewise triangulation of another $(k-1)$-simplex (see
 \cite{ED} for more details) and a suitable affine transformation. We describe
this triangulation directly.

A permutation $\pi=\pi_1\pi_2\ldots\pi_{k-1}\in \mathbb{S}_{k-1}$
is \emph{consistent} with a vertex $\bm{v}\in
V_{k,q}$ if $v_i>0$ whenever $i$ precedes $i-1$ in $\pi$. A simplex
in the triangulation of $\Delta_{k,q}$ is determined by
a vertex $\bm{v}\in V_{k,q}$ with $v_k>0$ and a consistent
permutation $\pi$. For any such pair, the corresponding simplex is
$e(\bm{v},\pi)=\mathrm conv \{\bm{v}^{(1)},\bm{v}^{(2)}
,\ldots,\bm{v}^{(k)}\}$, where
\begin{equation}\label{E:verticesofe}
\bm{v}^{(1)}=\bm{v},\bm{v}^{(2)}=\bm v^{(1)} 
+e_{\pi_{k-1}}-e_{\pi_{k-1}+1},\ldots,\bm{v}^{(k)}= \bm{v}^{(k-1)}
+e_{\pi_{1}}-e_{\pi_{1}+1}.
\end{equation}
Observe that for
$\pi=Id$ we obtain $e(\bm{v},Id)=e(\bm{v}-e_{k})$. We omit the
proof that $\{e(\bm{v},\pi)\}$ triangulates $\Delta_{k,q}$, since
this follows directly from the construction in \cite{MV}.

\begin{proposition}
Consider the dominant labeling of the triangulation of
$\Delta_{k,q}$ defined above. The simplex
$e((0,1,\ldots,1,q-k+2),k-1\;k-2\ldots 2\;1)$ is the unique fully
labeled simplex.
\end{proposition}

\begin{proof} Let
$e(\bm{v},\pi)=\{\bm{v}^{(1)},\bm{v}^{(2)},\ldots,\bm{v}^{(k)}\}$
be a fully labeled simplex, whose existence follows from
Sperner's lemma.
 For every $j\in[k-1]$, there exists $i\in[k]$
 such that $\ell(\bm v^{(i)})=j+1$, and consequently
 $v^{(i)}_j=0$.

It follows from (\ref{E:verticesofe}) that $v_j^{(i)}\leq 1$ for
all $j\in[k-1]$ and $i\in[k]$. Furthermore, if $\bm v^{(t)}$ is
 labeled with 1, then $\bm v^{(t)}=(1,1,\ldots,1,q-k+1)$.
 Combining these facts, we can
 conclude that $\bm v=(0,1,\ldots,1,q-k+2)$
 and $\pi=k-1\;k-2\ldots 2\;1$. Note that
 the point $(1,1,\ldots,1,q-k+1)$ is the unique apex
of the cones determined by the dominant labeling of $V_{k,q}$
(see Figure~\ref{konusi}).
\end{proof}

For each $\pi\in\mathbb{S}_{k-1}$, we define a $k$-uniform
hypergraph $H^\pi_{k,q}$ as follows. Its hyperedges are the sets
of vertices of the simplices $e(\bm v,\pi)$, where $\bm v\in
V_{k,q}$ satisfies $v_k>0$ and $\pi$ is consistent with $\bm v$.
The vertex set of $H^\pi_{k,q}$ consists of all vertices appearing
in these hyperedges. We use the same notation for a simplex and
the corresponding hyperedge.

In Section 2, we described the possible colors of the vertices of a hyperedge in
the Simplex-Lattice hypergraph $H_{k,q}=H^{Id}_{k,q}$ under the
dominant labeling. Here, we analyze the colors of the vertices of a hyperedge
in $H^{\pi}_{k,q}$ for a fixed $\pi\in \mathbb{S}_{k-1}$. We determine 
the change in the colors of the vertices of
$e(\bm{v},\pi)$ listed in (\ref{E:verticesofe}).

Assume that $\bm{v}\in V_{k,q}$ is labeled with $b$, that
is, $\ell(\bm{v})=b$. If $b=1$ then $\ell(\bm{v}+e_p-e_{p+1})=1$
for all $p\in [k-1]$. For $b>1$, let $a=a(\bm v)<b$ be the minimal
index such that $v_a+v_{a+1}+\cdots+v_{b-1}=b-a-1$. Such an index
$a$ always exists because $v_{b-1}=0$, and we say that $a$ is the
\emph{falling index} of $\bm v$. Further, if there exist an index $c>b$
such that $v_b+v_{b+1}+\cdots+v_{c-1}=c-b$, let $c=c(\bm v)$ be the minimal
such index, and call it the \emph{rising index} of $\bm v$.
\begin{lemma}\label{L:change} Assume that $\ell(\bm v)=b$
for a vertex $\bm v\in V_{k,q}$. For any $p\in [k-1]$, such that
$v_{p+1}>0$, we have
\[
\ell(\bm v+e_p-e_{p+1})=\begin{cases}
    b, & \text{if } p\neq b-1;\\
    a(\bm v), & \text{if } p=b-1 \text{ and } c(\bm v) \text{ does not exist};\\
    c(\bm v), & \text{if } p=b-1 \text{ and } c(\bm v) \text{ exists}.
\end{cases}
\]
Moreover, if $\ell(\bm v+e_p-e_{p+1})=a(\bm v)$, then $c(\bm v+e_p-e_{p+1})\leq b$.
\end{lemma}
\begin{proof}
The result follows directly from the definition of the dominant
labeling and the choice of the indices $a$ and $c$.
\end{proof}

\begin{theorem}\label{T:opada}
For $\bm v\in V_{k,q}$, a hyperedge $e(\bm v, \pi)$ receives the colors $b_0=1<b_1+1<b_2+1<\cdots<b_m+1$
under the dominant labeling if and only if:
\begin{enumerate}
    \item[(i)] $b_m,b_{m-1},\ldots, b_1$ is a decreasing subsequence of $\pi$; and
    \item[(ii)] $b_1-1$ does not appear after $b_1$ in $\pi$, and $b_i-1$ does not
    appear between $b_i$ and $b_{i-1}$ in $\pi$ for all $i>1$.
\end{enumerate}
\end{theorem}

\begin{proof} Assume $e(\bm v, \pi)$ 
received the colors $b_0=1<b_1+1<b_2+1<\cdots<b_m+1$ under the dominant
labeling. By Lemma~\ref{L:change}, the colors appear on the vertices 
of $e(\bm v,\pi)$ listed in (\ref{E:verticesofe}) in the order
$b_1+1,b_2+1,\ldots,b_m+1,1$.  Hence, $b_m,b_{m-1},\ldots,b_1$ is a
subsequence of $\pi$. Since $\ell(\bm v)=b_1+1$ we have $v_{b_1}=0$, which 
implies that $b_1-1$ is not after $b_1$ in $\pi$. Let $\bm v^{(t)}$ be the first
vertex of $e(\bm v,\pi)$ that received the color $b_i+1\ (i>1)$. We have
$\bm v^{(t)}=\bm v^{(t-1)}+e_{b_{i-1}}-e_{b_{i-1}+1}$ and $v^{(t)}_{b_i}=0$. This implies 
that $b_i-1$ is not between $b_i$ and $b_{i-1}$ in $\pi$.

To prove the converse, we are looking for $\bm v=(v_1,v_2,\ldots,v_k)$
such that $\pi$ is consistent with $\bm v$ and vertices of $e(\bm
v,\pi)$ have the colors $b_0=1<b_1+1<b_2+1<\cdots<b_m+1$. 

For $j\in [k-1]\setminus \{1,b_1+1,\ldots,b_m+1\}$, let
$v_j=1$ if $j$ precedes $j-1$ in $\pi$ and
$v_j=0$ if not. Choose the coordinates $v_1,v_{b_1+1},\ldots,v_{b_m+1}$ so
that the following hold:
\begin{gather}
v_1+v_2+\cdots +v_{b_1}=b_1-1,
 v_{b_m+1}+\cdots+ v_{k}=q-b_m+1,\label{e:sumeprvazadnja}\\
v_{b_i+1}+v_{b_i+2}+\cdots
 + v_{b_{i+1}}=b_{i+1}-b_{i}\textrm{ for all }1<i<m.\label{E:sumeost}
\end{gather}
If $k\neq b_m+1$, set $v_k=1$.

Observe that $\ell(\bm v)=b_1+1$. The equalities from  (\ref{e:sumeprvazadnja}) and
(\ref{E:sumeost}) provide that the rising index of
a vertex colored by $b_i+1$ is $b_{i+1}$ for $i<m$. If $i=m$, then the vertices
colored by $b_m+1$ do not have a rising index and by the first equation in (\ref{e:sumeprvazadnja}) their falling index is 1.
By Lemma~\ref{L:change}, we can conclude that $e(\bm v,\pi)$ receives the colors 
$b_0=1<b_1+1<b_2+1<\cdots<b_m+1$.
\end{proof}

\noindent The proof of the next theorem is similar to the previous
one, so we omit it.

\begin{theorem}\label{T:douledown}
For $\bm v\in V_{k,q}$, a hyperedge $e(\bm v,\pi)$ receives the colors $b_1+1<b_2+1<\cdots<b_m+1$
under the dominant labeling if and only if there exists $i>1$ such that the following conditions hold:
\begin{enumerate}
    \item[(i)] $b_{i-1},\ldots,b_1,b_m,\ldots,b_i$ is a subsequence of $\pi$ where  
		$b_{i-1},\ldots,b_1$ and $b_m,\ldots,b_i$ are both decreasing and $b_{i-1}<b_i$;
    \item[(ii)]  In $\pi$, $b_i-1$ is not after $b_i$ and $b_1-1$ is not between $b_1$ and $b_m$;
    \item[(iii)] For all $j$ satisfying $i<j\leq m$ or $1<j\leq i-1$, $b_j-1$ does not lie between $b_j$ and $b_{j-1}$ in $\pi$.
\end{enumerate}
\end{theorem}

Theorems \ref{T:opada} and \ref{T:douledown} give a complete
characterization of the color sets of hyperedges in
$H^{\pi}_{k,q}$ under the dominant labeling.

We say that a subsequence 
$\pi_{t_1},\ldots,\pi_{t_i},\pi_{t_{i+1}},\ldots, \pi_{t_m}$ of
$\pi$ is \emph{layered} if
\[\pi_{t_1}>\cdots>\pi_{t_i},\,\, \pi_{t_m}>\pi_{t_1}\textrm{ and }
 \pi_{t_{i+1}}>\cdots >\pi_{t_m}.\]

For a permutation $\pi$, let $LDS(\pi)$ and $LLS(\pi)$ denote the lengths of the longest decreasing
subsequence and the longest layered subsequence of $\pi$, respectively.
\begin{corollary} The maximum number of colors appearing
on a hyperedge of $H^{\pi}_{k,q}$ under the dominant labeling is
$\max \{LLS(\pi),1+LDS(\pi)\}$.
\end{corollary}

\bibliographystyle{plain}

\end{document}